\documentclass[10pt]{article}
\usepackage{amsmath,amsthm,amssymb,mathrsfs}
\usepackage{upgreek}
\usepackage{enumerate}
\usepackage{graphics,tikz-cd}
\usepackage{float}
\usepackage[T1]{fontenc}
\usepackage[textwidth=3.4cm,textsize=small]{todonotes}
\usepackage[normalem]{ulem}

\renewcommand\ge\geqslant
\renewcommand\geq\geqslant
\renewcommand\le\leqslant
\renewcommand\leq\leqslant

\numberwithin{equation}{section}

\newcommand{\refpart}[1]{{\it (#1)}}  

\newcommand{\PP}{\mathbb{P}}

\newcommand{\RR}{\mathbb{R}}

\newcommand{\DD}{{\mathcal D}}
\newcommand{\DI}{{\mathcal I}}
\newcommand{\DR}{{\mathcal R}}

\newcommand{\cc}{\lambda} % {\upkappa}
\newcommand{\FG}{\mathcal{F}}

\newtheorem{theorem}{Theorem}[section]
\newtheorem{lemma}[theorem]{Lemma} 
\newtheorem{propose}[theorem]{Proposition}

\newtheorem{remark}[theorem]{Remark}

\newcommand{\pcoor}[1]{%
	\begingroup\lccode`~=`: \lowercase{\endgroup
		\edef~}{\mathbin{\mathchar\the\mathcode`:}\nobreak}%
	(% opening symbol
	\begingroup
	\mathcode`:=\string"8000
	#1%
	\endgroup 
	)% closing symbol
}

\title{\bf Dupin cyclides passing through a fixed circle}

\author{
	Jean Michel Menjanahary$^a$\; and\; Raimundas Vidunas$^b$ \\
	\em $^a$Institute of Computer Science, Vilnius University\\
	\em $^b$Institute of Applied Mathematics, Vilnius University
}

\date{\empty}

\begin{document}
	\maketitle
	
	%%%%%%%%%%%%%%%%%%%%%%%%%%%%%%%%%%%%%%%%%%%%%%%%%%%%%%%%%%%%%%%%%%%%%%%%%%%%%%%%%%%%%%%%%%%
	\begin{abstract}
		We derive algebraic equations on the coefficients of the implicit equation to characterize all Dupin cyclides passing through a fixed circle.
		The results are applied to solve the basic problems in CAGD about blending of Dupin cyclides along circles.
	\end{abstract}
	
	%%%%%%%%%%%%%%%%%%%%%%%%%%%%%%%%%%%%%%%%%%%%%%%%%%%%%%%%%%%%%%%%%%%%%%%%%%%%%%%%%%%%%%%%%%%
	\section{Introduction}
	
	Dupin cyclides are remarkable algebraic surfaces that have applications % attract attention 
	in Computer Aided Geometric Design (CAGD). %They are useful surfaces in computer aided geometric design.
	The prototypical example of a Dupin cyclide is a torus of revolution with major radius $R$ and minor radius $r$.
	A canonical implicit equation of a torus is
	\begin{equation}
		\label{eq:torus}
		\big(x^2+y^2+z^2+R^2-r^2\big)^2-4R^2(x^2+y^2)=0.
	\end{equation}
	We must have $r<R$ for a smooth torus surface. 
	Any torus contains two orthogonal circles through each point.
	These circles are curvature lines of the torus, and are called {\em principal circles}.
	A smooth torus has two additional circles through each point  on a bitangent plane to the torus; see Figure \ref{fig:torus}\refpart{a}.
	They are called {\em Villarceau circles} \cite{Vill}.
	\begin{figure}
		\begin{center}
			\begin{picture}(320,100)
				\put(10,0){\includegraphics[width=4.2cm]{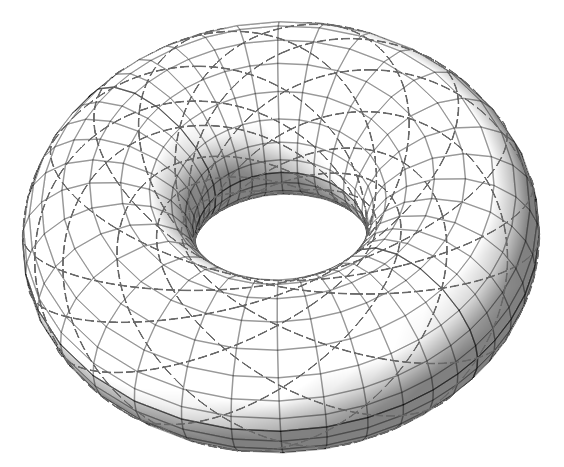}}
				\put(180,0){\includegraphics[width=4.5cm]{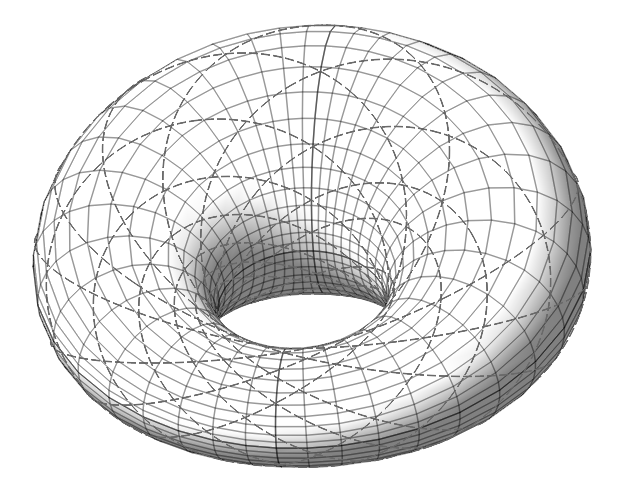}}
				\put(10,-4){\refpart{a}}
				\put(180,-4){\refpart{b}}
			\end{picture}
		\end{center}
		\caption{A smooth torus \refpart{a} and a smooth Dupin cyclide \refpart{b}. 
			The solid circles are principal circles and the dashed circles are Villarceau circles.}
		\label{fig:torus}
	\end{figure}
	A Dupin cyclide is the image of a torus under a M\"obius transformation, in particular an orthogonal transformation
	or an inversion with respect to a sphere. 
	These transformations preserve the angles and the set of circles and lines on the surfaces \cite{MV, Ottens}.
	Accordingly, smooth Dupin cyclides inherit the property of having 2 principal circles 
	and 2 Villarceau circles through each point; see Figure \ref{fig:torus}\refpart{b}.
	Some of these circles may degenerate to straight lines.
	Dupin cyclides are used in CAGD and architecture, predominantly for blending surfaces along the circles to model more elaborate CAGD surfaces 
	\cite{Martin, DuPratt1, DuPratt2, Zube, Kra, Lionel, Architecture, Salkov},
	or smooth blending them with natural quadrics and canal surfaces along the circles  \cite{Blend1, Blend2, Druoton}.
	
	The implicit equation for a Dupin cyclide is of degree 4 or 3 and has the form
	\begin{align} \label{eq:gendarb}
		a_0\big(x^2+y^2+z^2\big)^2&+2(b_1x+b_2y+b_3z)\big(x^2+y^2+z^2\big)\nonumber\\
		&+c_1x^2+c_2y^2+c_3z^2+2d_1yz+2d_2xz+2d_3xy\\
		&+2e_1x+2e_{2}y+2e_{3}z+f_{0} =  0, \nonumber
	\end{align}
	where $a_0,b_1,\ldots,f_0 \in \RR$. 
	For general values of the coefficients, this implicit equation defines a more general surface called {\em Darboux cyclide} \cite{DarCyc}.
	The practical problem of distinguishing Dupin cyclides among Darboux cyclides is considered in \cite{MV}.
	
	The basic problem considered in this paper is smooth blending of two Dupin cyclides along a fixed circle.
	Our approach is to match implicit equations (\ref{eq:gendarb}) for the two Dupin cyclides we blend.
	To solve the basic problem algebraically, we first consider the general linear family of Darboux cyclides
	passing through a fixed circle. Then we use the results in \cite{MV} to characterize the smaller family of Dupin cyclides  % passing through a fixed circle 
	in terms of the algebraic relations for the free coefficients of the general family of Darboux cyclides.
	This is considered in Section  \ref{sec:cagda} together with the formulation of the main results of the paper. We prove them separately for quartic and cubic equations in Sections \ref{sec:split} and \ref{sec:proofs}.
	The smooth blending between two implicit equations of  Dupin cyclides along the fixed circle is investigated in Section \ref{sec:join}.  In the last section, we express the M\"obius invariant of a Dupin cyclide equation, introduced in \cite{MV}, to our particular smaller families of Dupin cyclides.
	
	%%%%%%%%%%%%%%%%%%%%%%%%%%%%%%%%%%%%%%%%%%%%%%%%%%%%%%%%%%%%%%%%%%%%%%%%%%%%%%%%%%%%%%%%%%%
	\section{Preliminaries} 
	\label{sec:prelim}
	
	First off, let us recall the salient results in \cite{MV} on distinguishing Dupin cyclides among Darboux cyclides.
	They are formulated using the following abbreviations
	of algebraic expressions in the coefficients in (\ref{eq:gendarb}):
	\begin{align*}
		B_0 = & \; b_1^2+b_2^2+b_3^2,\\
		C_0 = & \; c_1+c_2+c_3,\\ 
		E_0 = & \; e_1^2+e_2^2+e_3^2,\\
		W_1 = &\;  c_1c_2+c_1c_3+c_2c_3-d_1^2-d_2^2-d_3^2,\\
		W_2 = &\; c_1c_2c_3+2d_1d_2d_3-c_1d_1^2-c_2d_2^2-c_3d_3^2, \\
		W_3 = &\; b_1^2c_1 + b_2^2c_2 + b_3^2c_3 + 2b_2b_3d_1 + 2b_1b_3d_2 + 2b_1b_2d_3,\\
		W_4 = &\;c_1e_1^2+c_2e_2^2+c_3e_3^2+2d_{1} e_{2} e_{3} + 2d_{2} e_{1} e_{3}  + 2 d_{3} e_{1} e_{2}.
	\end{align*}
	Let $\sigma_{12}, \sigma_{13}$ denote the permutations of the variables $b_1,b_2,b_3; c_1,c_2,c_3; d_1,d_2,d_3$ and $e_1,e_2,e_3$ that permute the indices $1,2$ or $1,3$, respectively.
	
	To recognize quartic Dupin cyclides among the form (\ref{eq:gendarb}), we can assume $a_0=1$ by dividing all coefficients by $a_0$. 
	Then we apply the shift 
	\begin{equation} \label{eq:p3shift} \textstyle
		(x,y,z)\mapsto (x,y,z) - \frac{1}{2}(b_1,b_2,b_3)
	\end{equation}
	to remove the cubic terms and reduce the equation to an intermediate Darboux form
	\begin{align} \label{eq:gendarb1}
		\big(x^2+y^2+z^2\big)^2&+c_1x^2+c_2y^2+c_3z^2+2d_1yz+2d_2xz+2d_3xy\\
		&+2e_1x+2e_{2}y+2e_{3}z+f_{0} =  0. \nonumber
	\end{align}
	\begin{theorem}
		\label{thm:main}
		The hypersurface in $\RR^3$ defined by $(\ref{eq:gendarb1})$ is a Dupin cyclide 
		only if the following $12$ polynomials vanish:
		% {\small
			\begin{align*}
				K_1 = & \; (c_3-c_2)e_2e_3+d_1(e_2^2-e_3^2)+(d_2e_2-d_3e_3)e_1,\\
				%K_2 &= (c_3-c_1)e_1e_3+d_2(e_1^2-e_3^2)+(d_1e_1-d_3e_3)e_2, \\ 
				%K_3 &= (c_1-c_2)e_1e_2+d_3(e_2^2-e_1^2)+(d_2e_2-d_1e_1)e_3,\\
				L_1 = & \; \big(W_1+4f_0-(c_2+c_3)^2-d_2^2- d_3^2\big)e_1 \\
				& +\!\big(C_0d_3+ c_3d_3- d_1d_2\big)e_2 +\big(C_0d_2+ c_2d_2- d_1d_3\big)e_3,\\
				%L_2 &= 
				%\big(W_1\!+\!4f_0\!-\!(c_1\!+\!c_3)^2\!\!\!-\! d_1^2\!-\! d_3^2\big)e_2 \!+\!\big(C_0d_3\!+\! c_3d_3\!-\! d_1d_2\big)e_1 \!+\!\big(C_0d_1\!+\! c_1d_1\!-\! d_2d_3\big)e_3,\\
				%L_3 &= 
				%\big(W_1\!+\!4f_0\!-\!(c_1\!+\!c_2)^2\!\!\!-\! d_2^2\!-\! d_1^2\big)e_3 \!+\!\big(C_0d_1\!+\! c_1d_1\!-\! d_2d_3\big)e_2 \!+\!\big(C_0d_2\!+\! c_2d_2\!-\! d_1d_3\big)e_1,\\
				M_1 = &  \; 2(c_1e_1+d_3e_2+d_2e_3)(W_1+4f_0)+e_1(W_2-C_0W_1-4E_0),\\
				%M_2 &=  2(c_2e_2+d_3e_1+d_1e_3)(W_1+4f_0)+e_2(W_2-C_0W_1-4E_0),\\
				%M_3 &=  2(c_3e_3+d_1e_2+d_2e_1)(W_1+4f_0)+e_3(W_2-C_0W_1-4E_0),\\
				N_1 = & \; \big(4W_1+12f_0-3 C_0^2\big)(W_1+4f_0) - 2C_0(W_2-C_0W_1-6E_0)-4\,W_4, \\
				N_2 = & \; 4(W_2\!-\!C_0W_1\!-\!2E_0)(W_1\!+\!4f_0) +\big(C_0^2\!-\!4f_0\big) \big(W_2\!+\!C_0W_1\!+\!8C_0f_0\!-\!4E_0\big), \\
				N_3 = & \; \big(\, W_2+C_0W_1+8C_0f_0-4E_0\big)^{\!2}-4(W_1+4f_0)^3,
			\end{align*}
			and $\sigma_{12}K_1,\;\sigma_{12}L_1,\;  \sigma_{12}M_1, \sigma_{13}K_1,\;\sigma_{13}L_1,\; \sigma_{13}M_1$.
		\end{theorem}
		\noindent{\em Proof.} 
		This result is covered in \cite[Proposition 3.6]{MV}. 
		We consider the necessity only because in this paper
		we start from the $12$ polynomial equations to find a Dupin cyclide.
		In addition, the $12$ polynomial equations define the following degenerate Dupin cyclides: two touching spheres, a sphere and a point on it, a double sphere, a circle, a cyclide with exactly 2, 1 or no real points; see \cite[Section 6.3]{MV}. We need to discard them from the algebraic equations.
		\qed 
		\begin{theorem}
			
			\label{thm:mainp}
			The hypersurface in $\RR^3$ defined by $(\ref{eq:gendarb})$ is a cubic Dupin cyclide only if
			the following equations are satisfied:
			\begin{align}\label{eq:cubic_result}
				a_0=&\, \textstyle 0, \qquad e_1= 
				\frac{1}{4}\,E_1, \qquad e_2=\frac{1}{4}\,\sigma_{12} E_1, \qquad e_3=\frac{1}{4}\,\sigma_{13} E_1,\\[2pt]
				f_0= &\, \frac{W_3}{4B_0^2}\left(\frac{W_3}{B_0}-C_0\right)^{\!2} +\frac{W_3W_1}{4B_0^2}+ \frac{W_2-C_0W_1}{4B_0},
			\end{align}
			where 
			{\small
				\begin{align*}
					E_1 = 
					&\;  -\frac{b_1}{B_0} \! \left(\frac{W_3}{B_0}-c_2-c_3\right)^{\!2} 
					+ \frac{2b_1^2}{B_0^2}\,(b_3c_3d_2 + b_2c_2d_3)
					- \frac{4b_1}{B_0^2}\,(b_3d_2 + b_2d_3)^2 \nonumber \\
					&\; + \frac{2(b_3d_2 + b_2d_3)}{B_0^2}\,(b_2^2c_1 + b_3^2c_1-2b_2b_3d_1)
					- \frac{2b_2b_3}{B_0^2}\,(c_2 - c_3)(b_2d_2 - b_3d_3)\nonumber\\
					&\;+\frac{b_1}{B_0}\,\big((c_1 - c_2)(c_1 - c_3)- d_1^2 + d_2^2 + d_3^2\big)
					+\frac{2d_1}{B_0}\,(b_2d_2 + b_3d_3).
				\end{align*}
			}
		\end{theorem}  
		\noindent{\em Proof.} 
		This is covered in \cite[Theorem 2.4]{MV}. Degenerate cases are a sphere touching a plane or a plane and a point on it; see \cite[Section 6.3]{MV}.\qed
		
		%%%%%%%%%%%%%%%%%%%%%%%%%%%%%%%%%%%%%%%%%%%%%%%%%%%%%%%%%%%%%%%%%%%%%%%%%%%%%%%%%%%%%%%%%%%
		\section{The main results}
		\label{sec:cagda}
		
		Without loss of generality, we assume that a fixed circle $\Gamma\subset \RR^3$ 
		with radius $r>0$, is given by the equations
		\begin{equation} \label{eq:circle}
			x=0, \qquad   y^2+z^2=r^2.
		\end{equation}
		Computing the variety of Dupin cyclides passing through the circle $\Gamma$ turns out to be non-trivial. 
		The defining equations are obtained by restricting the coefficients of (\ref{eq:gendarb}) to 
		cyclides passing through $\Gamma$ and 
		by considering the effects on the equations in Theorems \ref{thm:main} and \ref{thm:mainp}.
		The Darboux cyclides passing through the circle $\Gamma$ form a linear subspace of the space of coefficients in (\ref{eq:gendarb}).
		
		\begin{lemma}
			\label{lem:circleDarb}
			A Darboux cyclide passing through the circle $\Gamma$ has implicit equation of the form
			\begin{align} \label{eq:ldarb2}
				u_0(x^2+y^2+z^2-r^2)^2
				&+2(x^2+y^2+z^2-r^2)(u_1x+u_2y+u_3z+u_4) \nonumber\\
				&+ 2x\big(v_1x+v_2y+v_3z+v_4\big) =  0,
			\end{align}
			where $[u_0,\ldots,u_4;v_1,\ldots,v_4]\in \PP^8$.
		\end{lemma}
		
		\noindent{\em Proof.}
		The equation of a Darboux cyclide passing through the circle $\Gamma$ is an algebraic combination of $x$ and $x^2+y^2+z^2-r^2$.  
		Besides, the terms of degree four and degree three should match to the Darboux form (\ref{eq:gendarb}). We expand the quartic and cubic terms to
		\[
		u_0(x^2+y^2+z^2-r^2)^2+2(x^2+y^2+z^2-r^2)(u_1x+u_2y+u_3z),
		\]
		so that they would be contained in the ideal $(x,x^2+y^2+z^2-r^2)$ of the polynomial ring  $\RR(r)[x,y,z]$ over $\RR(r)$ -- the fraction field of $\RR[r]$. 
		The remaining terms of degree $\le 2$ should be in the same ideal, hence they have the shape
		\[
		2u_4(x^2+y^2+z^2-r^2)+ 2x\,\big(v_1x+v_2y+v_3z+v_4\big). \vspace{-6.5mm}
		\]
		\qed\\
		
		The ambient space Darboux cyclides passing through the circle $\Gamma$ is then identified as $\PP^8$, 
		with the coordinates $[u_0,\ldots,u_4;v_1,\ldots,v_4]$.
		The Dupin cyclides with real points will form an algebraic variety $\DD_{\Gamma}$ in this projective space.
		If we would consider the radius $r$ as a variable, the variety $\DD_{\Gamma}$ 
		would be invariant under the scaling of $(x,y,z)\in\RR^3$, and its equations would also be weighted-homogeneous, 
		with the weight 1 for $r$ and the respective weights $0,1,1,1,2,2,2,2,3$ of the coordinates of $\PP^8$.
		We rather consider $r$ as a parameter, and assume $r\neq 0$. 
		
		As it turns out, the variety $\DD_{\Gamma}$ is reducible, 
		reflecting the fact that the circle $\Gamma$ could be either a principal or Villarceau circle on a Dupin cyclide.
		Accordingly, we split the main result into two theorems as follows.
		
		\begin{theorem}\label{th:m3}
			The hypersurface in $\RR^3$ defined by $(\ref{eq:ldarb2})$ is a non-degenerate Dupin cyclide 
			containing $\Gamma$ as a Villarceau circle if and only if the % following 
			equations % are satisfied:
			\begin{align}  \label{eq:vv2v3}
				v_4-2r^2u_1 & =0, & v_1+2 u_4-2r^2u_0 &=0,\\   \label{eq:vv2v3a}
				u_2 v_2+u_3 v_3 - 2u_1u_4 & =0, & 
				4r^2(u_{1}^{2}+u_{2}^{2} + u_{3}^{2})-4u_4^2-v_2^2-v_3^2 & =0, % \nonumber
			\end{align}
			and % restricted furher by 
			the inequality
			\begin{equation} \label{eq:noneq}
				u_4^2<r^2(u_2^2+u_3^2)
			\end{equation}
			are satisfied.
		\end{theorem}
		
		\begin{theorem}\label{th:m3p}
			The hypersurface in $\RR^3$ defined by $(\ref{eq:ldarb2})$ is a non-degenerate Dupin cyclide 
			containing $\Gamma$ as a principal circle % if and 
			only if the ranks of the following two matrices are equal to $1$:
			% both the rank of the $9\times 2$ matrix
			\begin{align}  \label{eq:mmg}
				{\cal N} = & \left( \begin{array}{cc} 
					u_2 & v_2\\
					u_3 & v_3\\
					u_4 & v_4
				\end{array} \right), \\
				{\cal M} = & \left( \begin{array}{cc} 
					u_2 & v_2(v_4-2r^2u_1)\\
					u_3 & v_3(v_4-2r^2u_1)\\
					u_4 & v_4(v_4-2r^2u_1)\\
					2u_0 & v_2^2+v_3^2-4r^2u_1^2\\
					u_1 & 4r^2u_0v_4-2r^2(u_2v_2+u_3v_3)-4r^2u_1(v_1+u_4)\\
					% v_1 & 4r^4(2u_0v_1-u_1^2+u_2^2+u_3^2)-4r^2\left((v_1+u_4)^2-v_4u_1\right)-v_4^2\\
					v_1 & 4r^4(u_2^2+u_3^2+2u_0v_1)-4r^2(v_1+u_4)^2-(v_4-2r^2u_1)^2 \! \\
					v_2 & -8r^4u_1u_2-4r^2v_2(v_1+u_4-2r^2u_0)\\
					v_3 & -8r^4u_1u_3-4r^2v_3(v_1+u_4-2r^2u_0)\\
					v_4 & -8r^4u_1u_4-4r^2v_4(v_1+u_4-2r^2u_0)
				\end{array} \right) \!.
			\end{align}
		\end{theorem}
		\begin{remark} \rm
			The rank conditions mean vanishing of the $2\times 2$ minors of the matrices ${\cal N}$ and ${\cal M}$.
			The $2\times 2$ minors from the first 3 rows of ${\cal M}$ differ from the
			minors of ${\cal N}$ by the common factor $v_4-2r^2u_1$. Incidentally, this factor appears as an equation for the Villarceau case.
			Localizing with $(v_4-2r^2u_1)^{-1}$ leads to the ideal for the principal circle case. 
			But the Villarceau case equations of Theorem \ref{th:m3} do not imply a lesser rank of ${\cal M}$,
			as the second column does not necessarily vanish fully, particularly 
			in the fourth row. Rather similarly, the $2\times 2$ minors from the last 3 rows differ of ${\cal M}$ from the 
			minors of ${\cal N}$ by the common factor $-8r^4u_1$, as the terms $-4r^2v_i(v_1+u_4-2r^2u_0)$ are proportional to the first column.
			Therefore, the $2\times 2$ minors formed only by the first 3 rows or only by the last 3 rows of ${\cal M}$ can be ignored.
		\end{remark}

		\begin{remark} \rm %[Hilbert series, dimension, degree] \rm
			Using {\sf Maple} or {\sf Singular}, the Hilbert series of the ideal principal circle component is given by
			$H_{p}(t)/(1-t)^4$, where
			\begin{equation}
				H_{p}(t) = 1+4t+7t^2-10t^3+10t^4-5t^5+t^6.
			\end{equation}
			Hence, the dimension of the variety equals 4 and the degree equals $H_{p}(1)=8$.
			The Hilbert series of the ideal of Villarceau circle component is computed as $H_{v}/(1-t)^4$, where
			\begin{equation}
				H_{v}(t)=1+2t+t^2.
			\end{equation}
			It follows that the dimension of the variety equals $4$ and the degree equals $4$.
			The Zariski closure of the Villarceau circle component is a complete intersection.
		\end{remark}
		
		%%%%%%%%%%%%%%%%%%%%%%%%%%%%%%%%%%%%%%%%%%%%%%%%%%%%%%%%%%%%%%%%%%%%%%%%%%%%%%%%%%%%%%%%%%%
		\section{Distinguishing principal and Villarceau circles}
		\label{sec:split}
		
		As we will discuss in Section \ref{sec:proofs}, the variety $\DD_{\Gamma}$ turns out to be reducible.
		We discard some of the components because they either represent only degenerate reducible cases
		of cyclides, or the cases with complex (rather than real) coefficients in (\ref{eq:ldarb2}). 
		Theorems \ref{th:m3} and \ref{th:m3p} describe two components of $\DD_{\Gamma}$ 
		that have real points representing
		non-degenerate Dupin cyclides. 
		They are distinguished by the homotopy class of $\Gamma$ 
		as either a principal circle  or a Villarceau circle. 
		These two homotopical types can be distinguished by inspecting the type of $\Gamma$
		on representative % circular 
		torus surfaces. %like (\ref{eq:torus}).
		Indeed, principal circles are preserved by \cite[Theorem 3.14]{Ottens} by M\"obius transformations (finite composition of inversions). Hence, the components of $\DD_{\Gamma}$ % with real points 
		are invariant under M\"obius transformations fixing the circle $\Gamma$.
		The real components of $\DD_{\Gamma}$ contain subvarieties that correspond to toruses
		as representative Dupin cyclides under the continuous action of M\"obius transformations.
		
		Under Euclidean transformations, we move the torus equation (\ref{eq:torus}) so that the circle $\Gamma$ is a principal circle (with radius $r$) or a Villarceau circle \mbox{(with radius $R$)}.
		The principal circles on the vertical plane $x=0$ are given by $(y\pm R)^2+z^2=r^2$. 
		Identifying one of those circles with $\Gamma$ by the shift $y\mapsto y+R$,
		we obtain an equation of the form (\ref{eq:ldarb2})
		with
		\begin{align} \label{eq:uprinc}
			[u_0,u_1,u_2,u_3,u_4;v_1,v_2,v_3,v_4]  =  [1,0,-2R,0,2R^2;-2R^2,0,0,0]
		\end{align}
		for the representative (under the M\"obius transformations) tori with $\Gamma$ as a principal circle.
		It is straightforward to check that the representative tori (\ref{eq:uprinc}) 
		do not satisfy the second the fourth equations of Theorem \ref{th:m3} generically,
		while the second column of the matrix ${\cal M}$ in Theorem \ref{th:m3p} consists of zeroes for them.
		
		Now consider
		a Villarceau circle on $z=\alpha x + \beta y$ where $\alpha=r/\varrho,\beta=0$, $\varrho = \sqrt{R^2-r^2}$ 
		is moved onto $\Gamma$ % the circle $x=0$, $y^2+z^2=R^2$ 
		by the Euclidean transformation
		\begin{equation}
			(x,y,z)\mapsto \left(\frac{rx+\varrho z}{R},r-y,\frac{rz-\varrho x}{R}\right).
		\end{equation}
		Then the torus equation becomes
		\begin{equation} \label{eq:vtorus}
			\big(x^2+y^2+z^2-2ry+R^2\big)^2-4 \big( (rx+\varrho z)^2+R^2(y-r)^2\big)=0.
		\end{equation}
		This identifies (\ref{eq:ldarb2}) with
		\begin{align} \label{eq:uvillar}
			[u_0,u_1,u_2,u_3,u_4;v_1,v_2,v_3,v_4] =  [1,0,-2r,0,2r^2;2R^2-4r^2,0,-4r\varrho,0] 
		\end{align}
		as an implicit equation for the representative tori with $\Gamma$ as a Villarceau circle.
		The representative tori (\ref{eq:uvillar}) satisfy the equations of Theorem \ref{th:m3},
		while the rows with $u_2$ and $u_0$ in the first column form a lower-triangular matrix
		with non-zero determinant generically. Therefore, Theorem \ref{th:m3} describes cases
		with $\Gamma$ as a Villarceau circle, and Theorem \ref{th:m3p} describes cases
		with $\Gamma$ as a principal circle, as claimed. 
		
		\begin{remark}\rm
			\label{rem:skipPR}
			The variety $\DD_{\Gamma}$ contains a non-interesting big component consisting of touching spheres and touching sphere/plane cases.
			This component is defined by the $2\times 2$ minors of the matrix
			\begin{align}  \label{eq:mmg1}
				{\cal L} = & \left( \begin{array}{cc} 
					u_2 & v_2\\
					u_3 & v_3\\
					u_4 & v_4\\
					u_0v_2 & 2(u_1v_2-u_2v_1)\\
					u_0v_3 &  2(u_1v_3-u_3v_1)\\
					u_0v_4 &  2(u_1v_4-u_4v_1)\\
				\end{array} \right),
			\end{align}
			and the equation
			\begin{align}
				\label{eq:touch-cond2}
				4r^2(u_1^2+u_2^2+u_3^2)+v_2^2+v_3^2-8v_1(r^2u_0-u_4)-4v_4u_1-4u_4^2 = 0.
			\end{align}
			%The touching spheres or sphere/plane in this variety are 
			%\begin{align}
			%&x^2 + y^2 + z^2 + sx - r^2=0,\\
			%&u_0(x^2 + y^2 + z^2) + (2u_1-su_0)x + 2u_2y + 2u_3z + 2u_4-r^2u_0=0,
			%\end{align}
			%where $s=v_i/u_i$, $i=1,2$ or $3$.
			%The touching point $(x,y,z)$ is
			%\[
			%(x,y,z)=
			%-\frac{\big(s(u_2^2+u_3^2-2u_0u_4)+2u_1u_4,u_2(su_1-2v_1+2u_4),u_3(su_1-2v_1+2u_4)\big)}
			%{2\big(u_1^2+u_2^2+u_3^2-2u_0v_1\big)}.
			%\]
			%
			In addition, the degeneration to the circle $\Gamma$ satisfy ${\rm rank}({\cal L})=0$ and with the two more equations: $u_1=0$ and $v_1=2r^2u_0$. 
			Furthermore,  the principal component restricted to  ${\rm rank}({\cal L})=0$ only has double sphere cases.
			Therefore, the matrix ${\cal L}$ has to be of full rank $2$ in order to obtain a non-degenerate Dupin cyclide in the principal circle component.
		\end{remark}
		
		\begin{remark}\rm
			\label{rm:skipVL}
			In the Villaceau circle component, we must have \mbox{$u_4^2\leq r^2(u_2^2+u_3^2)$}. 
			Otherwise, there is no real solution for $v_2$ in (\ref{eq:vv2v3a}). 
			The strict inequality $(\ref{eq:noneq})$ throws away horn cyclides ($J_0=0$ with reference to Section \ref{sec:invariant}) among Dupin cyclides defined by the main equations $(\ref{eq:vv2v3})-(\ref{eq:vv2v3a})$. Those horn Dupin cyclides are also in the principal circle component. Note that the intersection between the touching spheres component and the principal circle component represents touching spheres with touching point on the circle $\Gamma$. The touching spheres component intersects the Villarceau component at a sphere containing $\Gamma$ and a point on $\Gamma$. The latter intersection is contained in the principal component.
		\end{remark}
		
		%%%%%%%%%%%%%%%%%%%%%%%%%%%%%%%%%%%%%%%%%%%%%%%%%%%%%%%%%%%%%%%%%%%%%%%%%%%%%%%%%%%%%%%%%%%
		\section{Proving Theorems $\ref{th:m3}$ and $\ref{th:m3p}$}
		\label{sec:proofs}
		
		Let us define 
		the ring 
		\begin{equation}
			\DR_{\Gamma} = \RR(r)[u_1,u_2,u_3,u_4,v_1,v_2,v_3,v_4],
		\end{equation}
		and let us denote the $2\times 2$ minors ${\cal N}$ by
		\begin{align}
			T_2= &\, u_3v_4-u_4v_3, \\ 
			T_3= &\, u_2v_4-u_4v_2, \\  
			T_4= &\, u_2v_3-u_3v_2.
		\end{align}
		Let us also denote
		\begin{align}
			U_0=u_1^2+u_2^2+u_3^2.
		\end{align}
		
		We split the proofs into two cases for quartic and cubic Dupin cyclides by the use of Theorems \ref{thm:main} and \ref{thm:mainp} in a parallel way.
		We arrive at parallel options to simplify the reducible variety $\DD_{\Gamma}$ from the full consideration of equations in those theorems.
		Most of the considered particular equations or factors appear naturally by considering the mentioned eliminations and localizations.
		%Formally, the proof does not have to justify consideration of particular polynomials.
		
		%%%%%%%%%%%%%%%%%%%%%%%%%%%%%%%%%%%%%%%%%%%%%%%%%%%%%%%%%%%%%%%%%%%%%%%%%%%%%%%%%%%%
		\subsection{Proof for quartic cyclides}
		
		Without loss of generality, we may assume $u_0=1$ while considering quartic cyclides. 
		To apply Theorem \ref{thm:main}, it is necessary to apply the shift (\ref{eq:p3shift})
		with $(b_1,b_2,b_3)=(u_1,u_2,u_3)$,
		so to bring the cyclide equation (\ref{eq:ldarb2}) to the form (\ref{eq:gendarb1}).
		The obtained expression is
		\begin{align}\label{eq:translatedFC}
			& \left(x^2+y^2+z^2\right)^2
			+\left(2(u_4 +v_1-r^2)-u_1^2-\frac{U_0}{2}\right)x^2\nonumber\\
			& + \left(2(u_4-r^2)-u_2^2-\frac{U_0}{2}\right)y^2 % \nonumber\\
			+ \left(2(u_4-r^2)-u_3^2-\frac{U_0}{2}\right)z^2\nonumber\\
			& - 2 u_2u_3 yz + 2(v_3 - u_1 u_3)xz + 2(v_2 - u_1 u_2)xy \\
			& - \left(2u_1v_1+u_2v_2+u_3v_3-2v_4-u_1(U_0-2u_4)\right) x\nonumber\\
			&  - \left(u_1v_2-u_2(U_0-2u_4)\right)y 
			- \left(u_1v_3-u_3(U_0-2u_4)\right)z\nonumber\\
			& - \frac{3U_0^2}{16} + \frac{U_0(u_4+r^2)+u_1(u_1v_1+u_2v_2+u_3v_3-2v_4)}{2} - 2r^2u_4 + r^4 = 0.\nonumber
		\end{align}
		% where \begin{align}
			% U_0=u_1^2+u_2^2+u_3^2. \end{align}
		Identification with the coefficients $c_1,c_2,\ldots,f_0$ in (\ref{eq:gendarb1}) defines % a substitution for 
		the ring homomorphism
		\[
		\rho:\RR[c_1,c_2,c_3,d_1,d_2,d_3,e_1,e_2,e_3,f_0]\rightarrow \DR_{\Gamma}. % = \RR(r)[u_1,\ldots,v_4]
		\]
		Let $\DI_{\Gamma}\subset \DR_{\Gamma}$ 
		denote the ideal generated by the  $\rho$ \!\!-images of the 12 polynomials in Theorem \ref{thm:main}.
		The polynomials in this ideal have to vanish when (\ref{eq:ldarb2}) is a Dupin cyclide.
		Our goal is to describe real points representing non-degenerate Dupin cyclides on this variety. 
		
		The polynomial % immediately 
		$\rho(K_1)$ factors in $\DR_{\Gamma}$, namely $\rho(K_1)=-\frac14T_4V_0$, where
		\begin{align*}
			V_0= &\; u_1^2\,(2u_1u_4\!-\!u_2v_2\!-\!u_3v_3)\!+\!(u_2^2\!+\!u_3^2\!-\!2u_4)(2u_1u_4\!+\!2u_1v_1\!+\!u_2v_2\!+\!u_3v_3\!-\!2v_4).
		\end{align*}
		This shows that the variety defined by $\DI_{\Gamma}$ is reducible. 
		% Hence, \todo{prescribe???} one can prescribe that $T_4=V_0=0$ describes the principal circle component.
		%One can immediately check that $\rho(K_1)$ is divisible by $T_4$ in $\DR_{\Gamma}$.
		To investigate real points of the variety, we consider the possible three options: $T_4\neq0$; $V_0\neq 0$; and  $T_4=V_0=0$.
		
		Assume that $T_4\neq0$.  Elimination of
		$v_2,v_3,v_4$ % from $\DI_{\Gamma}$ 
		gives the product $V_1V_2\in\DI_{\Gamma}$ in the remaining variables, where 
		\begin{align}
			V_1 =  \; v_1+2u_4-2r^2, \quad V_2 = \; (u_1^2+u_2^2+u_3^2-2u_4)^2+4r^2u_1^2.
		\end{align}
		If $V_2=0$, then $U_0-2u_4=0$, $u_1=0$ as we look only for real components.
		The augmented ideal contains this sum of squares: \mbox{$v_4^2+r^2V_1^2=0$.}
		Therefore, $V_1=0$ is inevitable for the real components with $T_4\neq 0$. % With $V_1=0$, 
		The ideal \mbox{$\DI_{\Gamma}+(V_1)$} in $\DR_{\Gamma}[T_4^{-1}]$ contains several
		multiples of the polynomial $V_3=v_4-2r^2u_1$. 
		Localizing $V_3\neq 0$ gives the trivial ideal of $\DR_{\Gamma}[T_4^{-1},V^{-1}_3]$, hence an empty variety.
		With $V_3=0$ we obtain the equations of Theorem \ref{th:m3} in the homogenized form with $u_0$.
		The points on the corresponding variety describe cases when $\Gamma$ is a Villarceau circle, as analyzed in the previous section. % \ref{sc:
			
			Secondly, assume that $V_0\neq 0$. Localization of $\DI_{\Gamma}$ in the ring $\DR_{\Gamma}[V_0^{-1}]$ gives an ideal 
			generated by the $2\times 2$ minors of the matrix ${\cal L}$ in (\ref{eq:mmg1})
			and the additional equation (\ref{eq:touch-cond2}) with $u_0=1$. Hence we only obtain degenerate Dupin cyclides according to Remark \ref{rem:skipPR}.
			
			% \item 
			The last option is $T_4 = V_0 = 0$. We notice polynomial multiples of $T_2^2+T_3^2$ in the Gr\"obner basis of $(\DI_{\Gamma},T_4,V_0)$.
			Localization at $T_2^2+T_3^2\neq 0$ gives an ideal that contains the 4 polynomials of Theorem \ref{th:m3}. 
			Hence, it describes some points in the Villarceau circle component (of the option $T_4\neq0$). 
			We assume further that $T_2=T_3=0$.  % It is left to simplify the ideal $\DI_{\Gamma}^+=(\DI_{\Gamma},T_2,T_3,T_4,V_0)$. 
			% The strategy now is to keep specializing the ideal $\DI_{\Gamma}^+$ by certain key polynomials as far as we cover all cases. 
			Consideration of the following polynomial allows further progress: %  simplifies the ideal $\DI_{\Gamma}^+$:
			\begin{align}
				V_4 = &\;(2r^2u_1+v_4)(U_0-2u_4-2v_1)-u_1(4r^2u_4+v_2^2+v_3^2)\nonumber\\
				&\;+(v_1-4r^2)(u_2v_2+u_3v_3)+8r^2v_4.
			\end{align}
			The localization $V_4\neq 0$ leads to a subcase (describing touching spheres) of the option $V_0\neq0$. 
			Hence, we assume that $V_4=0$.
			Elimination of  
			$v_2,v_3,v_4$ in the ideal  $(\DI_{\Gamma},T_2,T_3,T_4,V_0,V_4)$
			leads to  some generators that factor with 
			\begin{equation}
				V_5=u_1^2(u_2^2+u_3^2)+(u_2^2+u_3^2-2u_4)^2.
			\end{equation}
			The further localization $V_5\neq0$ leads to the principal circle component in Theorem \ref{th:m3p}.
			The remaining case $V_5=0$ splits into these two subcases, as we are interested in the real points only:
			\begin{enumerate}
				\item[\refpart{i}] $u_1\neq 0$, so that $u_2=u_3=u_4=0$. 
				The obtained ideal is reducible, with the prominent factor $V_6=u_1^2(v_2^2+v_3^2)+4v_4^2$
				after elimination of $v_1$. The localization $V_7\neq0$ belongs to the principal circle component. 
				The case $V_6=0$ simplifies to $v_2=v_3=v_4=2v_1-u_1^2=0$, and the cyclide degenerates to a double sphere case. 
				\item[\refpart{ii}] $u_1=0$, $u_2^2+u_3^2-2u_4=0$. Elimination of the variables $u_1$, $u_2$, $u_3$, $u_4$ gives us a principal ideal,
				and the generator factors with 
				\begin{equation}
					V_7=(v_2^2+v_3^2)^3+(v_1v_2^2+v_1v_3^2+2v_4^2)^2. 
				\end{equation}
				The localization $V_7\neq0$ belongs to the principal circle component. 
				With $V_7=0$ we get $v_2=v_3=v_4=0$, and the resulting ideal contains the product $(u_2^2+u_3^2+2v_1)^2(u_2^2+u_3^2+2v_1-4r^2)$. 
				Either of the factors leads to points on the principal circle component. 
			\end{enumerate}
			
			%%%%%%%%%%%%%%%%%%%%%%%%%%%%%%%%%%%%%%%%%%%%%%%%%%%%%%%%%%%%%%%%%%%%%%%%%%%%%%%%%%%%%%%
			\subsection{Proof for cubic cyclides}
			
			We use Theorem  \ref{thm:mainp} to recognize cubic Dupin cyclides in the form  (\ref{eq:ldarb2}) with $u_0=0$. The equation is first transformed to the form (\ref{eq:gendarb})
			\begin{align*}
				&2(u_1x+u_2y+u_3z)(x^2+y^2+z^2)\\ 
				&+ (2u_4+2v_1)x^2+2u_4y^2+2u_4z^2+2v_2xy+2v_3xz\\
				&+2(v_4-r^2u_1)x-2r^2u_2y-2r^2u_3z-2r^2u_4=0.
			\end{align*}
			% aa,bb,cc,dd,ee,ff,gg,hh,jj,kk,ll,mm,nn=u1,u2,u3,2*u4+2*v1,2*u4,2*u4,v2,v3,0,v4-R*u1,-R*u2,-R*u3,-2*R*u4;
			%Those cubic Dupin cyclides are defined by 
			%the 4 modified equations of Theorem \ref{thm:mainp}, under the natural ring homomorphism mapping the coefficients of $(\ref{eq:gendarb})$ to the coefficients of (\ref{eq:ldarb2}) in the ring $\DR_{\Gamma}$.
			Let 
			\[
			\rho_0:\RR[b_1,b_2,b_3,c_1,c_2,c_3,d_1,d_2,d_3,e_1,e_2,e_3,f_0]\rightarrow \DR_{\Gamma}. % = \RR(r)[u_1,\ldots,v_4]
			\]
			be the ring homomorphism defined by the coefficients identification. 
			Since $\rho_0(B_0)=U_0$, all remaining computations will be considered over the localized ring $\DR_{\Gamma}[U_0^{-1}]$. 
			Let us denote by $\DI_{\Gamma}^*$ the ideal generated by the numerators of the $\rho_0$-images of the 4 equations in Theorem \ref{thm:mainp}. 
			This ideal contains the product $T_4V^*_0$, where
			\begin{align}
				V^*_0=2u_1u_4U_0+2u_1v_1(u_2^2+u_3^2)+(u_2v_2+u_3v_3)(u_2^2+u_3^2-u_1^2).
			\end{align}
			Similar to the simplification method in the quartic case, we consider the three options $T_4\neq 0$;  $V_0^*\neq 0$; and  $T_4=V_0^*=0$.
			
			The localization $T_4\neq 0$ gives us directly the $u_0=0$ part of the Villarceau circle component in Theorem \ref{th:m3}.
			
			The localizing $V_0^*\neq 0$ gives an ideal containing the $2\times 2$ minors of the matrix ${\cal L}$. Hence, this case describes only degenerate cyclides, see Remark \ref{rem:skipPR}. 
			
			With $T_4=V_0^*=0$, the ideal $(\DI_{\Gamma}^*,T_4,V_0^*)$ contains the sum of squares $T_2^2+T_3^2$. 
			Hence $T_2=T_3=0$ since we are looking for real points of the variety $\DD_{\Gamma}$ only.
			The further candidate for localization to consider is
			\[V_1^* = 4r^2u_1^2+v_2^2+v_3^2-4u_1v_4.\]
			By comparing Gr\"oebner bases, the localization  of $(\DI_{\Gamma}^*,T_2,T_3,T_4,V_0^*)$ at $V_1^*\neq 0$ indeed coincides with the ideal of principal circle defined by the $2\times 2$ minors of ${\cal N}$ and ${\cal M}$. 
			The remaining case $V_1^*=0$ can be localized further at $V_2^*=u_2^2+u_3^2+u_4^2$. 
			The localization $V_2^*\neq 0$ defines points on the principal circle component. 
			The case $V_2^*=0$ simplifies to $u_2=u_3=u_4=0$ and the cyclide equation degenerates to a subcase of touching sphere/plane case.
			
			%%%%%%%%%%%%%%%%%%%%%%%%%%%%%%%%%%%%%%%%%%%%%%%%%%%%%%%%%%%%%%%%%%%%%%%%%%%%%%%%%%%%%%%%%%%
			\section{Smooth blending of cyclides}
			\label{sec:join}
			
			Here we apply the main results to the practical problem of blending smoothly two Dupin cyclides along a common circle.
			Smooth blending in this context means that the cyclides share tangent planes along their common circle. 
			
			\begin{lemma}
				\label{th:joincond}
				Consider two cyclides equations of the form $(\ref{eq:ldarb2})$ with possibly different coefficients $[u_0,\ldots,u_4,v_1,\ldots,v_4]$. Then they are joined smoothly along the circle $\Gamma$ if and only if
				the rational function 
				\begin{equation} \label{eq:g1join}
					\FG(y,z)=\frac{v_2y+v_3z+v_4}{u_2y+u_3z+u_4}
				\end{equation}
				is the same function on the circle $\Gamma$ for both cyclides.
			\end{lemma}
			\noindent{\em Proof.} 
			%\todo{The normal vector is defined by the gradient}
			The normal vector of cyclides  (\ref{eq:ldarb2}) along the circle $\Gamma$ is defined by % transposition of 
			the gradient of the defining polynomial. The gradient is computed as
			\[
			\big( v_2 y + v_3 z + v_4, 2y (u_2y+u_3z+u_4), 2z (u_2y+u_3z+u_4) \big).
			\]
			On the two given cyclides, the paired gradient vectors should be proportional along the circle 
			in order to obtain smooth blending.  After the division by $u_2y+u_3z+u_4$, 
			the gradient vectors are rescaled to 
			% \[
			$\big(\FG(y,z), 2y, 2z\big)$
			% \]
			for direct comparison.
			\qed \\
			
			A special case is when the rational function (\ref{eq:g1join}) is a constant on $\Gamma$. This is equivalent to ${\rm rank}({\cal N})=1$.
			Therefore, the rational function $\FG$ is constant when $\Gamma$ is a principal circle case of a Dupin cyclide. 
			As the following lemma implies, the envelope surface of tangent planes of any cyclide equation satisfying  ${\rm rank}({\cal N})=1$ along $\Gamma$ is a circular cone or cylinder.
			It is known \cite{Kra} that the envelope appearing as a cone or cylinder occurs in the case of Dupin cyclides if the circle is principal. 
			This is due to the representation of Dupin cyclides as canal surfaces, where they are considered as conics in 4-dimensional Minkowski space and the tangent lines to those conics represent circular cones or cylinders; see \cite{Kra} for details. 
			\begin{lemma}
				\label{th:F12const}
				The function $\FG(y,z)\equiv \cc$ on the circle $\Gamma$ for some constant $\cc$ 
				if and only if the envelope surface of tangent planes of the cyclide $(\ref{eq:ldarb2})$ along $\Gamma$ is 
				given by the equation
				\begin{align}
					\label{eq:coneFamily}
					y^2+z^2=\left(r-\frac{\cc x}{2r}\right)^2.
				\end{align}
				It is a circular cone if $\cc\neq 0$, or a cylinder if $\cc=0$. 
			\end{lemma}
			\begin{proof}
				We parametrize the circle by $(0,r\cos\varphi,r\sin\varphi)$.
				The envelope line passing through such a point is orthogonal to the rescaled gradient vector
				$\big(\cc,2r\cos\phi,2r\sin\phi\big)$ and to the tangent vector $(0,-\sin\phi,\cos\phi)$ to the circle.
				The line therefore follows the direction of the cross product vector
				$(2r,-\cc\cos\phi,-\cc\sin\phi)$. The envelope of tangent planes is parametrized therefore as  
				\begin{align}
					\label{eq:coneFamily0}
					(x,y,z)=(0,r\cos\varphi,r\sin\varphi) + t \, (2r,-\cc\cos\varphi,-\cc\sin\varphi).
				\end{align}
				Hence 
				%\begin{align}
				$x=2rt$, $y^2+z^2 = (r-\cc t)^2$.
				% \end{align}
			Elimination of $t$ gives (\ref{eq:coneFamily}).
			The reverse is also true because the family (\ref{eq:coneFamily})
			contains all circular cones passing through the circle $\Gamma$. 
			%The cylinder case corresponds to $\cc=0$.
		\end{proof}
		
		\begin{remark}\rm
			\label{rem:PrVl}
			The envelope of tangent planes degenerates to the plane of the circle $\Gamma$ when $\cc=\infty$. If the circle is a Villarceau circle, then the envelope of tangent planes is a much more complicated surface of degree 4.
		\end{remark}
		
		%%%%%%%%%%%%%%%%%%%%%%%%%%%%%%%%%%%%%%%%%%%%%%%%%%%%%%%%%%%%%%%%%%%%%%%%%%%%%%%%%%%%%%
		\subsection{Smooth blending along principal circles}
		
		In this section, we focus on smooth blending between Dupin cyclide equations in the principal circle component. 
		The main case to investigate is by fixing a tangent cone along the circle $\Gamma$ and find Dupin cyclides in the principal circle component that fit the blending conditions along the circle, see Figure \ref{fig:joints}\refpart{a}.
		
		%%%%%%%%%%%%%%%%%%%%%%%%%%%%%%%%%%%%%%%%%%%%%%%%%%%%%%%%%%%%%%%%%%
		\begin{propose}
			\label{thm:fixedfamily}
			Let us fix the parameter $\cc\neq 0$ and the cone $(\ref{eq:coneFamily})$. 
			The family of Dupin cyclides in the principal circle component touching the cone along the circle $\Gamma$ is defined by the $5$ equations
			\begin{align}
				&v_2=\cc u_2,\quad v_3= \cc u_3, \quad v_4=\cc u_4,\\
				&4r^2u_1(\cc u_0-u_1)+\cc^2(u_2^2+u_3^2)-2\cc u_0u_4=0,\label{eq:cc4}\\
				&16r^4(\cc u_0-u_1)^2+4\cc^2r^2u_1^2-\cc^2(\cc^2+4r^2)(u_2^2+u_3^2)-8\cc^2r^2u_0v_1=0.\label{eq:cc5}
			\end{align}
		\end{propose}
		\noindent{\em Proof.} 
		From Lemmas \ref{th:joincond}--\ref{th:F12const}, the tangency conditions along the circle are given by $v_i=\cc u_i$ for $i\in\{2,3,4\}$. 
		We specialize $u_0,v_2,v_3,v_4$ in the ideal generated by the $2\times2$ minors of ${\cal N}$ and ${\cal M}$, % Theorem \ref{th:m3p}
		and obtain an ideal $\DI_\lambda$ in $\DR_{\lambda}=\RR(r)[u_1,u_2,u_3,u_4,v_1,\lambda,\lambda^{-1}]$. 
		We notice many multiples of  $u_2,u_3,u_4$ in a Gr\"obner basis of $\DI_\lambda$. 
		% motivates us to consider the localisation $u_2u_3u_4\neq0$. 
		If $u_2u_3u_4\neq0$, we obtain an ideal $\DI^*_\lambda\subset \DR_{\lambda}[(u_2u_3u_4)^{-1}]$ 
		generated by the 5 equations of the proposition.
		The points with $u_2u_3u_4=0$ satisfy the equations of $\DI^*_\lambda\cup \DR_{\lambda}$,
		by checking the cases $u_2=u_3=u_4=0$, $u_i=0, u_ju_k\neq0$ or $u_i=u_j=0,u_k\neq 0$ with $i,j,k\in \{2,3,4\}$ pairwise distinct. 
		Each of the resulting ideals $\DR_{\Gamma}[\cc,\cc^{-1}]$ contains $\DI^*_\lambda\cup \DR_{\lambda}$. \qed\\
		
		\begin{remark} \rm 
			\label{rem:generalFamily}
			The five equations of Proposition \ref{thm:fixedfamily} are linear in the five variables $u_4$, $v_1$, $v_2$, $v_3$, $v_4$. 
			Hence, we can easily solve the equations in those variables and obtain a parametrization of the family of Dupin cyclides touching the cone along the circle $\Gamma$.
			Apart from the first 3 equations, the variables $u_2,u_3$ appear only within the expression $u_2^2+u_3^2$, representing a rotational degree of freedom:
			rotating the two Dupin cyclide patches independently around the $x$-axis preserves the smooth blending along the circle $\Gamma$.
		\end{remark}
		
		\begin{figure}
			\begin{center}
				\begin{picture}(300,300)
					\put(20,300){\includegraphics[width=4.5cm]{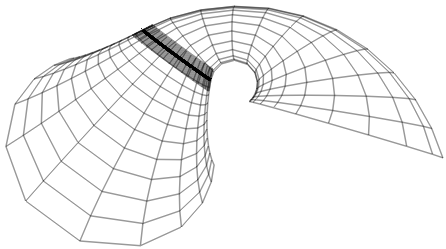}}  
					\put(200,300){\includegraphics[width=4cm]{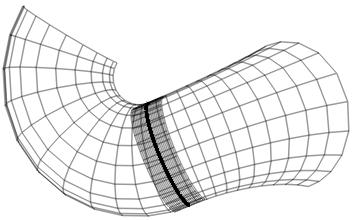}}
					\put(20,150){\includegraphics[width=4.5cm]{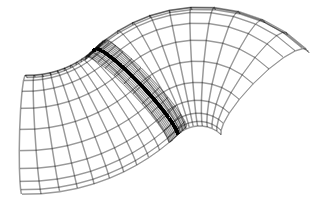}}
					\put(180,150){\includegraphics[width=4.5cm]{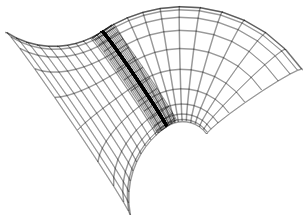}} 
					\put(0,0){\includegraphics[width=6cm]{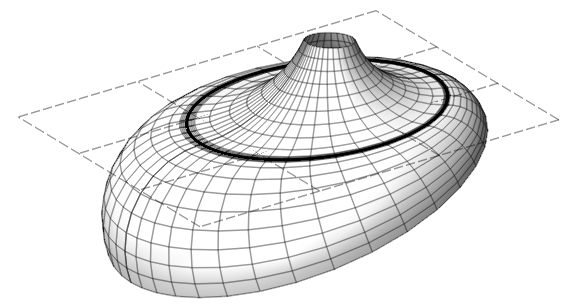}}
					\put(200,-5){\includegraphics[width=4cm]{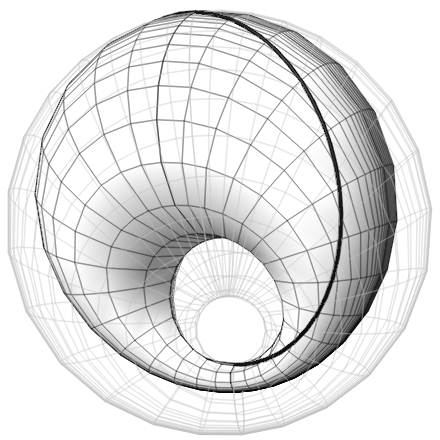}}
					\put(0,280){\refpart{a}}
					\put(15,280){$[1, -\frac{49}{30}, 0, \frac{76}{15}, \frac{323}{30};
						-\frac{1669}{120}, 0, -\frac{76}{15}, -\frac{323}{30}]$}
					\put(35,260){$[1, -2, -5, 0, \frac{17}{2}; -\frac{93}{8}, 5, 0, -\frac{17}{2}]$}
					\put(190,280){\refpart{b}}
					\put(205,280){$[1,0,-3,0,\frac{9}{2};\frac{-9}{2},0,0,0]$}
					\put(200,260){$[1,0,0,\frac{76}{15},\frac{323}{30}; -\frac{361}{30},0,0,0]$}
					\put(0,140){\refpart{c}}
					\put(20,140){$[\frac{17}{15},0,\frac{17}{3},0,\frac{85}{6};-\frac{85}{6},0,0,0]$}
					\put(28,125){$[1, 0, -3, 0, \frac{9}{2}; -\frac{9}{2}, 0, 0, 0]$}
					\put(190,140){\refpart{d}}
					\put(210,140){$[1, 0, 0, 0, -4; 8, 0, 0, 0]$}
					\put(205,125){$[1, 0, -3, 0, \frac{9}{2}; -\frac{9}{2}, 0, 0, 0]$}
					\put(5,-20){\refpart{e}}
					\put(30,-20){$[1, a, 0, 0, 0; \frac{4a^2+15}{8}, 1, 0, 2a]$}
					\put(160,-20){\refpart{f}}
					\put(175,-20){$[1+t, 0, 1, 0, \frac{12}{13}; \frac{2}{13}+2t, 0, -\frac{10}{13}, 0]$}
				\end{picture}   
			\end{center}\vspace{5mm}
			\caption{Two Dupin cyclide equations with different coefficient values $[u_0,\ldots,u_4;v_1,\ldots,v_4]$ are smoothly blended along the circle $\Gamma$ with $r=1$. The two cyclides in \refpart{e} are obtained from the parameter values $a=1$ and $a=1.8$. The two cases at \refpart{f} are obtained from the parameter values $t=0$ and $t=0.4$.}
			\label{fig:joints}
		\end{figure}
		
		The limit cases $\cc=0$ and $\cc=\infty$ of principal circle component also contain interesting families of Dupin cyclides. The family in the case $\cc=0$ allows us in particular to blend two torus equations or a torus and a Dupin cyclide, see Figure \ref{fig:joints}\refpart{b},\refpart{c},\refpart{d}. The family in the case  $\cc=\infty$ allows us in particular to blend a Dupin cyclide with a plane, see Figure \ref{fig:joints}\refpart{e}. 
		
		\begin{propose}
			\label{thm:cyl}
			Dupin cyclides touching the cylinder along the circle $\Gamma$ are defined by the equations
			\begin{align}
				& u_1=v_2=v_3=v_4=0, \label{eq:condCyl}\\
				& 2r^2u_0v_1+r^2(u_2^2+u_3^2)-(v_1+u_4)^2=0.\label{eq:cond2Cyl}
			\end{align}
			Dupin cyclides here are symmetric with respect to plane of the circle $\Gamma$. 
		\end{propose}
		\noindent{\em Proof.} 
		The equations $v_2=v_3=v_4=0$ follow from the condition $\cc=0$. 
		With those constraints, the ideal of principal circle component  reduces to the other two equations $u_1=0$ and (\ref{eq:cond2Cyl}). The symmetry property with the plane $x=0$ follows from equations (\ref{eq:condCyl}). \qed\\
		
		\begin{propose}
			\label{thm:Plane}
			Dupin cyclides touching the plane of the circle $\Gamma$ along the same circle are defined by
			\begin{align}
				& u_2=u_3=u_4=0, v_4=2r^2u_1,\\
				&16r^4u_0^2+4r^2u_1^2-(v_2^2+v_3^2)-8r^2u_0v_1=0.
			\end{align}
		\end{propose}
		\noindent{\em Proof.} 
		The proof is similar to the proof of Proposition \ref{thm:cyl}. The equations $u_2=u_3=u_4=0$ follow from the tangency condition $\lambda=\infty$ and the ideal of principal circles reduces to the other two equations.
		\qed\\
		
		\begin{remark}\rm
			\label{rem:Plane}
			A cubic cyclide equation among the family in Proposition \ref{thm:Plane} degenerates to a reducible surface (touching sphere/plane).   
		\end{remark}
		
		%%%%%%%%%%%%%%%%%%%%%%%%%%%%%%%%%%%%%%%%%%%%%%%%%%%%%%%%%%%%%%%%%%
		It is interesting to distinguish toruses in the principal circle component. 
		We get two cases, depending on the position of the circle $\Gamma$ (wrapping around the torus hole or around the torus tube). 
		Figure \ref{fig:joints}\refpart{c} and \refpart{d} illustrate two different configuration of torus blending using those two kinds of principal circles. The circle wraps around the torus tube of both toruses in Figure \ref{fig:joints}\refpart{c}. The circle wraps around the torus tube for one torus and around the torus hole for the other torus in Figure \ref{fig:joints}\refpart{d}.
		
		\begin{propose}
			The equation \eqref{eq:ldarb2} defines a torus in the principal circle component if and only if one of the following applies:
			\begin{enumerate}[(i)]
				\item $u_0=1,\; u_2^2+u_3^2=2u_0u_4,\; v_1=-u_4,\; v_2=v_3=v_4=0$;
				\item $u_0=1,\; u_2=u_3=v_2=v_3=0$, $\displaystyle\; u_4 = \frac{2r^2u_1(\cc-u_1)}{\cc^2}$, \\[3pt]
				$\displaystyle  v_1= \frac{\cc^2u_1^2+4r^2(\cc-u_1)^2}{2\cc^2}$, $\displaystyle\; v_4=\cc u_4=\frac{2r^2u_1(\cc-u_1)}{\cc}$.
			\end{enumerate}
		\end{propose}
		\noindent{\em Proof.} 
		Assume that the circle $\Gamma$ is wrapping around the torus tube. Then we have a tangent cylinder along the circle, encoded by $v_2=v_3=v_4=0$ as in Proposition \ref{thm:cyl}.
		The cross section of (\ref{eq:ldarb2}) with the plane $x=0$ is a pair of circles of the same radius $(\Gamma,\Gamma')$
		\[\Gamma': x=\left(y+\frac{u_2}{u_0}\right)^2+\left(z+\frac{u_3}{u_0}\right)^2-\frac{r^2u_0^2-2u_0u_4+u_2^2+u_3^2}{u_0^2}=0.\]
		Besides the equality between radiuses, we have $u_2^2+u_3^2=2u_0u_4$. Hence the equation \eqref{eq:cond2Cyl} factors into $(v_1+u_4)(v_1+u_4-2r^2u_0)$.
		To recognize a torus equation, we can assume that $u_3=0$ and $u_2=\sqrt{2u_0u_4}$ by applying rotation preserving the circle $\Gamma$. The rotated cyclide equation is
		\begin{equation}
			u_0 \! \left( \! x^2 \! + \! \left(y \!-\! \sqrt{\frac{u_4}{2u_0}}\right)^{\!2} \!+\! z^2 \!- r^2 \! +\frac{u_4}{2u_0}\right)^{\!\!2} \! 
			+ 2u_4 \! \left(y \!-\! \sqrt{\frac{u_4}{2u_0}}\right)^{\!2} \! +2x^2v_1=0.
		\end{equation}
		This is a torus only if $u_0=1$ and $v_1=-u_4$. This proves \refpart{i}.
		
		Assume now that the circle $\Gamma$ is wrapping around the torus hole. 
		Then we have a tangent cone along the circle, i.e. $v_2=\cc u_2$, $v_3=\cc u_3$, $v_4=\cc u_4$ as in Proposition \ref{thm:fixedfamily}. 
		The section with $x=0$ should be a pair of concentric circles. 
		Hence, $u_2=u_3=0$. 
		Again with $u_0=1$ and the parametrization in Proposition \ref{thm:fixedfamily}, the cyclide equation reduces to
		\begin{align*}
			\left(\left(x+\frac{u_1}{2}\right)^2+y^2+z^2
			+\frac{r^2(\cc-u_1)^2}{\cc^2}- \frac{u_1^2 (\cc^2+4r^2)}{4\cc^2}\right)^2\\
			-4\frac{r^2(\cc-u_1)^2}{\cc^2} (y^2+z^2)=0.  
		\end{align*}
		This is a torus equation, see (\ref{eq:torus}). \qed\\
		
		%%%%%%%%%%%%%%%%%%%%%%%%%%%%%%%%%%%%%%%%%%%%%%%%%%%%%%%%%%%%%%%%%%%%%%%%%%%%%%%%%%%%%%
		\subsection{Smooth blending along Villarceau circles}
		Due to different tangency conditions along the circle $\Gamma$, see Remark \ref{rem:PrVl}, it is not possible to smoothly blend Dupin cyclides from different components (principal and Villarceau). So it is left to investigate blending between cyclides in the Villarceau circle component. 
		The following result is illustrated in Figure \ref{fig:joints}\refpart{f}.
		
		\begin{propose}
			Assume that $(\ref{eq:ldarb2})$ defines a Dupin cyclide in the Villarceau circle component.
			Then the pencil obtained by perturbing $(\ref{eq:ldarb2})$ by 
			\[(x^2+y^2+z^2-r^2)^2+4r^2x^2\] 
			consists of Dupin cyclides in the same Villarceau circle component. Furthermore, elements of this pencil are the only Dupin cyclides in this component satisfying the smooth blending conditions with $(\ref{eq:ldarb2})$ along the circle $\Gamma$. 
		\end{propose}
		\noindent{\em Proof.}
		Let $D'=[u_0,u_1',\dots,u_4';v_1',\dots,v_4']$ be a Dupin cyclide touching the given Dupin cyclide $D$ along the circle $\Gamma$ in the Villarceau circle component. We obtain the following equations in matrix form:
		\begin{equation*}
			\begin{pmatrix}
				0  & 0  & 0  & 2  & 1  & 0  & 0  & 0 \\
				-2r^2  & 0  & 0  & 0  & 0  & 0  & 0  & 1 \\ 
				0  &  r^2v_2  & 0  & v_4  & 0  & -r^2u_2  & 0  & -u_4 \\ 
				0  & 0  &  r^2v_3  & v_4  & 0  & 0  & -r^2u_3  & -u_4 \\
				0  &  v_3  &  v_2  & 0  & 0  & - u_3  & - u_2  & 0 \\
				0  & v_4  & 0  &  v_2  & 0  &  -u_4  & 0  & -u_2\\
				0  & 0  & v_4  &  v_3  & 0  & 0  &  -u_4  & -u_3 
			\end{pmatrix}
			\begin{pmatrix}
				u_1'\\ u_2'\\ u_3'\\ u_4'\\
				v_1'\\ v_2'\\ v_3'\\ v_4'
			\end{pmatrix}
			=
			\begin{pmatrix}
				2 r^2u_0\\0\\ 0\\ 0\\
				0\\ 0\\ 0\\ 0
			\end{pmatrix}.
		\end{equation*}
		The first two rows of the matrix are linear equations obtained from $D'$ being in the Villarceau circle component. The last $5$ rows are the tangency conditions to the given Dupin cyclide $D$ from Lemma \ref{th:joincond}. Note that the 7 by 8 matrix has full rank. We must have $v_i\neq 0$ for some $i\in \{2,3,4\}$ due to degeneracy to horn cyclides. Then by setting $s=u_i'/v_i$, we can solve
		\begin{align}
			u_j' &= su_j,\quad v_j' = sv_j,\quad j\in \{2,3,4\},\\
			u_1' &=s\frac{v_4}{2r^2} = su_1, \quad v_1' = 2r^2u_0-2su_4.
		\end{align}
		Hence, dividing the equation of $D'$ by $s$, all coefficients are fixed except $v_1'=2r^2u_0/s-2u_4$ and $u_0$ becomes $u_0/s$. Hence with $t=u_0/s-u_0$, $u_0$ and  $v_1'$ become $u_0+t$ and $2r^2u_0-2u_4+2r^2t=v_1+2r^2t$ respectively.
		The obtained family is neccessary in the Villarceau circle component because only the second equation in Theorem \ref{th:m3} depends on $u_0$ and $v_1$, and it is preserved by the perturbation.
		\qed
		
		%%%%%%%%%%%%%%%%%%%%%%%%%%%%%%%%%%%%%%%%%%%%%%%%%%%%%%%%%%%%%%%%%
		\section{The M\"obius invariant $J_0$}
		\label{sec:invariant}
		
		In this section, we compute a M\"obius invariant denoted by $J_0$ \cite[Section 6]{MV} specifically for Dupin cyclides in the Villarceau and principal circle components. 
		This invariant extends the toric invariant 
		\begin{equation}
			J_0=\frac{r^2}{R^2}\left(1-\frac{r^2}{R^2}\right)
		\end{equation} 
		to any Dupin cyclide equation.
		The smooth Dupin cyclides are characterized by $0< J_0\leq 1/4$ and the singular Dupin cyclides are characterized by $J_0\leq 0$.
		Note that a singular Dupin cyclides can be obtained from a spindle or a horn torus (see Figure \ref{fig:torus-sing}) by M\"obius transformations.
		It is convenient to subtract $1/4$ from $J_0$ and obtain a perfect square on the remainder
		\begin{equation}
			J_0=\frac{1}{4}-\frac{(R^2-2r^2)^2}{4R^4}.    
		\end{equation}
		
		We use \cite[(6.15)]{MV} and \cite[(6.17)]{MV} to compute $J_0$ for the quartic equation (\ref{eq:ldarb2}) with $u_0\neq 0$  and the cubic equation (\ref{eq:ldarb2}) with $u_0= 0$ respectively. 
		The obtained expression necessarily gives the M\"obius invariant when the equation defines a Dupin cyclide. Let us denote by $\widehat{J}_0$ the remainder $1/4-J_0$ of  (\ref{eq:ldarb2}).
		The goal is to have a compact equivalent formula for $J_0$ in each of the two components.
		
		In the Villarceau circle component, by reducing both numerator and denominator of  $\widehat{J}_0$ with the ideal of this component generated by (\ref{eq:vv2v3})--(\ref{eq:vv2v3}), we obtain a common factor. The common factor is cancelled by the fraction and the invariant simplifies to
		\begin{align} \label{eq:villj0}
			J_0
			&=\frac{1}{4}-\frac{r^2v_1^2}{4\big(r^2(v_1^2+v_2^2+v_3^2)-v_4^2)}\\
			&=\frac{r^2u_2^2+r^2u_3^2-u_4^2}{16(r^2u_2^2+r^2u_3^2-u_4^2)+4v_1^2}.
		\end{align}
		The second expression is obtained from the first one by eliminating $v_2$, $v_3$ and $v_4$ from (\ref{eq:vv2v3})--(\ref{eq:vv2v3}). This invariant coincides with the invariant computed using the formula for the cubic cyclides. This invariant should be positive, because singular cyclides have no real Villarceau circles. 
		Indeed, the numerator is positive by the inequality (\ref{eq:noneq}). 
		The denominator is positive as well from the same condition.  
		As a limiting case of Theorem \ref{th:m3}, 
		Villarceau circles could be identified on the horn cyclides with $J_0=0$, 
		where they coincide with the (vertical) circles passing through the singularity.
		
		\begin{figure}
			\begin{center}
				\begin{picture}(320,100)
					\put(10,0){\includegraphics[width=4.1cm]{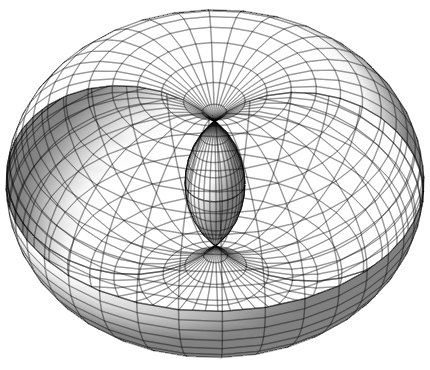}}
					\put(180,0){\includegraphics[width=4.5cm]{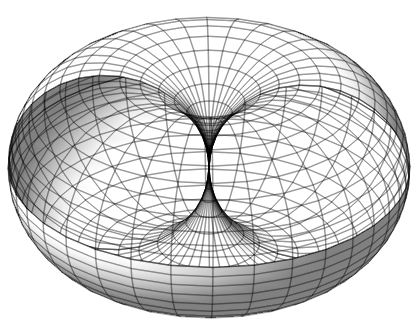}}
					\put(10,-4){\refpart{a}}  \put(180,-4){\refpart{b}}
				\end{picture}
			\end{center}
			\caption{A cutaway view of singular toruses: \refpart{a} a spindle torus ($J_0\!<0,r>R$);  
				\refpart{b} a horn torus ($J_0=0$, $r=R$).}
			\label{fig:torus-sing}
		\end{figure}
		
		In the principal circle component, we first consider quartic Dupin cyclides in Proposition \ref{thm:fixedfamily}. Consider the ideal $\DI_{\cc}$ generated by the 5 equations of the Proposition. By incorporating separately the numerator and the denominator of $\widehat{J}_0$ in the ideal $\DI_{\cc}$ and by eliminating the linear variables  $u_4,v_1,\ldots,v_4$, we obtain a representative numerator and a representative denominator with a common factor. This gives a new expression of $\widehat{J}_0$ up to a constant multiplier. It is easy to find this constant by solving it from the difference of the two expressions of $\widehat{J}_0$ modulo $\DI_{\cc}$. 
		The resulting $J_0$ expression is 
		\begin{align}
			J_0 = \frac{1}{4}-\frac{\big(8r^4(\cc u_0-u_1)^2-4r^2(\cc^2+4r^2)u_1^2+\cc^2(\cc^2+2r^2)(u_2^2+u_3^2)\big)^2}{16r^4\big(4r^2(\cc u_0-u_1)^2-\cc^2(u_2^2+u_3^2)\big)^2}.
		\end{align}
		By further elimination of $u_2^2+u_3^2$ using (\ref{eq:cc4})--(\ref{eq:cc5}), we obtain the more compact form
		\begin{align} \label{eq:j0genpc}
			J_0 = \frac{1}{4}-\frac{\big(4r^4\cc u_0-2r^2(\cc^2+6r^2)u_1+\cc(\cc^2+2r^2)u_4\big)^2}{16r^4\big(2r^2\cc u_0-2r^2u_1-\cc u_4\big)^2}.
		\end{align}
		It is interesting that this compact form (\ref{eq:j0genpc}) also covers the $J_0$ expression of the family of cubic Dupin cyclides $u_0=0$ in Proposition \ref{thm:fixedfamily}.
		
		Since the majority of Dupin cyclides in the principal circle component belong to the family of Dupin cyclides in Proposition \ref{thm:fixedfamily}, three equivalent expressions for $J_0$ in the principal circle component are obtained by substituting $\cc=v_i/u_i$ to (\ref{eq:j0genpc}) for each $i=2,3,4$. This coincidence also can be checked by reducing the numerator of the difference between the general $J_0$ and each of the found three expressions modulo the ideal of principal circle component.
		
		In the two limit cases Propositions \ref{thm:cyl} and \ref{thm:Plane} of principal circle component, we use the same method and obtain the following expression:
		
		\begin{align}
			J_0=\frac{1}{4}-\frac{(4r^2u_0-4u_4-3v_1)^2}{4v_1^2},
		\end{align}    
		for the family of Dupin cyclides in Proposition \ref{thm:cyl}, and
		\begin{align}
			J_0=\frac{1}{4}-\frac{(3r^2u_0-v_1)^2}{4r^4u_0^2},
		\end{align} 
		for the family of Dupin cyclides in Proposition \ref{thm:Plane}. Note that the latter formula is well-defined because the family in Proposition \ref{thm:Plane} does not contain non-degenerate cubic Dupin cyclides, see Remark \ref{rem:Plane}.
		
		%%%%%%%%%%%%%%%%%%%%%%%%%%%%%%%%%%%%%%%%%%%%%%%%%%%%%%%%%%%%%%%%%%%%%%%%%%%%%%%%%%%%%%
		\section*{Acknowledgments}
		%Thanks to Severinas Zub\.e and Rimvydas Krasauskas for useful remarks and suggestions.
		This work is part of a project that has received funding from the European Union’s Horizon 2020 research and innovation programme under the Marie Skłodowska-Curie grant agreement No 860843.
		
		%%%%%%%%%%%%%%%%%%%%%%%%%%%%%%%%%%%%%%%%%%%%%%%%%%%%%%%%%%%%%%%%%%%%%%%%%%%%%%%%%%%%%%
		\small
		\bibliographystyle{alpha}

	\end{document}